\documentclass[12pt]{article}
\usepackage{latexsym}
\usepackage{amssymb}
\usepackage{amsmath}

\usepackage{epsfig,graphicx,subfigure}
\usepackage{amsthm,amsmath,amssymb,amsfonts}
\usepackage{color,xcolor}
\usepackage[all]{xy}
\usepackage[top=30mm, bottom=30mm, left=35mm, right=35mm]{geometry}
\usepackage{hyperref}%The package hyperref provides LaTeX the ability to create hyperlinks within the document

\newtheorem{theorem}{Theorem}[section]
\newtheorem{lemma}[theorem]{Lemma}
\newtheorem{proposition}[theorem]{Proposition}

\theoremstyle{definition}

\newtheorem{example}[theorem]{Example}

\newtheorem{remark}[theorem]{Remark}
\numberwithin{equation}{section}
\begin{document}
	\thispagestyle{plain}
	
	\vspace*{9mm}
	
	\begin{center}
		
		{\Large \bf 
			Mean Ergodicity of Multiplication Operators on the Bloch and Besov Spaces}

		\bigskip

		{\bf  F.  Falahat}\vspace*{-2mm}\\
		\vspace{2mm} {\small Department of Mathematics, Faculty of Sciences, Shiraz Branch, Islamic Azad University,  Shiraz, Iran} \vspace{2mm}
		
		{\bf Z. Kamali$^*$\let\thefootnote\relax\footnote{$^*$Corresponding Author}}\vspace*{-2mm}\\
		\vspace{2mm} {\small   Department of Mathematics, Faculty of Sciences, Shiraz Branch, Islamic Azad University,  Shiraz, Iran} \vspace{2mm}
		
	\end{center}
	
	\vspace{4mm}
%\vspace{-15mm}
\begin{abstract}
%\normalsize % Add normalize right at the beginning of the abstract
In this paper, the power boundedness and mean ergodicity of multiplication operators are investigated on the Bloch space $\mathcal{B}$, the little Bloch space $\mathcal{B }_0 $ and the Besov Space $\mathcal{B}_p$. We completely characterize power bounded, mean ergodic and uniformly mean ergodic multiplication operators on $\mathcal{B}$ and $\mathcal{B }_0 $.
\end{abstract}

%\noindent\textbf{Mathematics Subject Classification (2010):} 47A16, 47B33, 47B38. 
\baselineskip=.8cm
\section{Introduction}
Let $\mathbb{U}=\{z \in \mathbb{C} : |z|<1\}$ be the unit disk in the complex plane $\mathbb{C}$ and $H(\mathbb{U})$ be the space of all holomorphic functions on $\mathbb{U}$. The Bloch space $\mathcal{B}$ is defined to be the space  of all functions in $H(\mathbb{U})$  such that
$$\beta_{f}=\sup_{z\in \mathbb{U}} (1-|z|^2)|f'(z)|<\infty.$$
The little Bloch space $\mathcal{B}_{0}$ is the closed subspace of $\mathcal{B}$ consisting of all functions $f\in \mathcal{B}$ with
$$\lim_{|z|\rightarrow 1}(1-|z|^{2})|f'(z)|=0.$$
It is easy to check that the Bloch and little Bloch, $\mathcal{B}$ and $\mathcal{B}_{0}$ are  Banach spaces under the norm 
$$||f||_{\mathcal{B}}=|f(0)|+\beta_{f}.$$\\
It is well known that $\mathcal{B}_{0}^{*}=A^{1}(\mathbb{U})$ and $(A^{1}(\mathbb{U}))^{*}=\mathcal{B}$ under the complex integral pairing $<f,g>=\int_{\mathbb{U}}f(z)\overline{g(z)}dA(z)$, where $dA(z)$ is lebesgue area measure on $\mathbb{U}$ and $A^{1}(\mathbb{U})$ is the space of all analytic functions $f$ on $\mathbb{U}$ such that $||f||=\int_{\mathbb{U}}|f(z)|dA(z)<\infty$, see \cite{zu2}.\\
Another space we dealt with in this paper is the  Besov  space  $\mathcal{B}_p$ $(1<p<\infty)$ which is defined to be the space of holomorphic functions $f$ on   $\mathbb{U}$ such that 
\begin{align*}
\gamma_f^p&=\int_{\mathbb{U}}|f'(z)|^p(1-|z|^2)^{p-2}dA(z)\\
&=\int_{\mathbb{U}}|f'(z)|^p(1-|z|^2)^pd\lambda(z)<\infty,
\end{align*}
where $d\lambda(z)$ is the M$\ddot{o}$bius invariant measure on U, with definition
$$d\lambda(z) = \frac{dA(z)}{(1-|z|^2)^2}.$$
For $p=1$ , the Besov space $\mathcal{B}_1$ consists of all holomorphic functions $f$ on $\mathbb{U}$  whose second derivatives are integrable,
$$\mathcal{B}_1=\{f\in H(\mathbb{U}) : ||f||_{\mathcal{B}_1}=\int_{\mathbb{U}}|f''(z)|dA(z)<\infty\}.$$
For $1<p<\infty$, $||f||_{p}=|f(0)| + \gamma_f$  is a norm on ${\mathcal{B}_p}$ which makes it a Banach space. $\mathcal{B}_p$ is reflexive space (while $\mathcal{B}_1$ is not) and polynomials are dense in it. Furthermore, for each $1<q<p<\infty$, $\mathcal{B}_1 \subset \mathcal{B}_q\subset \mathcal{B}_p \subset \mathcal{B}$ and $\mathcal{B}_1$ is a subset of the little Bloch space $\mathcal{B}_0$ (see  \cite{zu1}). Also remember that the Besov space $\mathcal{B}_2$  is  known as the classical $Dirichlet$ Space $\mathcal{D}$ and $\mathcal{B}_\infty$ is the Bloch space $\mathcal{B}$.
Moreover, the following two useful lemmas determine that norm convergence implies pointwise convergence in the Bloch and Besov spaces  $\mathcal{B}_p (1<p<\infty)$. We state them here without proof.
\begin{lemma}
For all $f\in\mathcal{B}$ and for each $z\in\mathbb{U}$, we have
$$|f(z)|\le||f||_\mathcal{B}\log\frac{2}{1-|z|^2}.$$
\begin{proof}
See \cite{zu2}.
\end{proof}
\end{lemma}
\begin{lemma}
For each $f\in\mathcal{B}_p$ $(1<p<\infty)$ and for every $z\in\mathbb{U}$, there is $C\ge0$ (depends only on p) such that 
$$|f(z)|\le C||f||_{p}(\log\frac{2}{1-|z|^2})^{1-1/p}.$$
\begin{proof}
See \cite{zu1}.
\end{proof}
\end{lemma}
Bloch and Besov spaces and their properties specially from an operator and geometric view were studied extensively in previous years  in \cite{Arz2, zu1, zu2} and more recently in \cite{Gir, Gal}. 

If $\psi$ is a holomorphic function on  $\mathbb{U}$, the $multiplication \ operator$  $M_{ \psi}$  on $H(\mathbb{U})$ is defined by
$$M_{\psi}(f)=\psi f.$$
We recall that for a set  $\Omega$, $H^\infty(\Omega)=\{f\in H(\Omega):||f||_{\infty}=\sup_{z\in\Omega}|f(z)|<\infty\}$.
\begin{proposition}
	Let $X$ be a functional Banach space on the set  $\Omega$ and suppose $\psi$ is a complex-valued function on $\Omega$ such that  $\psi X \subset X$. Then the operator $M_\psi$  is a bounded operator on $X$ and $|\psi(x)| \le ||M_\psi||$  for all $x\in \Omega$. In particular, $\psi \in H^\infty (\Omega)$.
\end{proposition}
\begin{proof}
	See \cite{kre}.
\end{proof}
According to Lemma 1.1. and Lema 1.2. all three spaces, $\mathcal{B}$, $\mathcal{B }_0 $ and $\mathcal{B}_p$ are functional Banach spaces and by above proposition in all of them we have the following inequality:
\begin{align}
||\psi||_\infty \le ||M_\psi||.\
\end{align}
Investigating the boundedness or compactness and other properties of multiplication operators on the Bloch and Besov spaces have been done by many authors (see \cite{Arz2, Aln, zor}). Arazy in \cite{Arz2} proved that the multiplication operator $M_\psi$  is bounded on the Bloch space if and only if $\psi\in H^\infty(\mathbb{U})$ and $ \sigma_{\psi}<\infty$ where 
$$\sigma_{\psi}:=\sup_{z\in\mathbb{U}}\frac{1}{2}(1-|z|^2)|\psi'(z)|\log{\frac{1+|z|}{1-|z|}}.$$
Also Brown and Shields in \cite{brown} proved that $M_\psi$ is bounded on Bloch space if and only if it is bounded on little Bloch space.
In \cite{Aln}, Allen and Colonna  showed that if  $\psi\in H(\mathbb{U})$  induces a bounded multiplication operator $M_\psi$ on the Bloch space, then
\begin{align}
\max\{||\psi||_\mathcal{B},||\psi||_\infty\}\le ||M_\psi||\le\max\{||\psi||_\mathcal{B},||\psi||_\infty+\sigma_{\psi}\}.\
\end{align}

In the case of Besov Space, the situation is somewhat different and the issue is not as straightforward as the Bloch Space. A function $\psi\in H(\mathbb{U})$ is said to be $multiplier$ of $\mathcal{B}_p$ if $M_\psi(\mathcal{B}_p)\subseteq \mathcal{B}_p$. If the space of multipliers on $\mathcal{B}_p$  in to itself represented by $M(\mathcal{B}_p)$, then by Closed Graph theorem  $\psi\in M(\mathcal{B}_p)$ if and only if $M_\psi$ is a bounded operator on $\mathcal{B}_p$.
Stegenga \cite{ste}, Characterized  multipliers of the Dirichlet space $\mathcal{D}$ in to itself. Characterization of multipliers of the Besov space $\mathcal{B}_p$ ($1 < p <\infty$), based on capacities and Carleson measures type conditions was given by Wu \cite{wu} and Arcozzi \cite{Arc}. Zorboska in \cite{zor}, Corollary 3.2, proved that for $1<p<\infty$ if $\psi\in M(\mathcal{B}_p)$, then $\psi\in H^\infty(\mathbb{U})$ and
 $$\sup_{z\in\mathbb{U}}(1-|z|^2)|\psi'(z)|(\log{\frac{2}{1-|z|^2}})^{1-1/p}<\infty .$$
Following proposition is another applied result of Zorboska about multiplication operators on the Besov spaces.
\begin{proposition}
Suppose that $1<p<\infty$ and $\psi\in H^\infty (\mathbb{U})$.
 \begin{itemize}
\item[(i)] If $\psi\in M(\mathcal{B}_p)$ and  $0 < r < 1$, then
$$\sup_{\omega\in D}\int_{D(\omega,r)}(1-|z|^2)^{p-2}|\psi'(z)|^p(\log{\frac{2}{1-|z|^2}})^{p-1}dA(z)<\infty,$$ where $D(\omega,r)= \{z\in \mathbb{U}: \beta(z,\omega)<r\}$ is the hyperbolic disk with radius r,  $\beta(z,\omega)= \log\frac{1+|\psi_{z}(\omega)|}{1-|\psi_{z}(\omega)|} $ and $\psi_{z}(\omega)=\frac{z-\omega}{1-\bar{z}\omega}$ for all $z , \omega\in\mathbb{U}.$
\item[(ii)] If $\int_{\mathbb{U}}(1-|z|^2)^{p-2}|\psi'(z)|^p(\log{\frac{2}{1-|z|^2}})^{p-1}dA(z)<\infty$, then $\psi\in M(\mathcal{B}_p)$.
\end{itemize}
\begin{proof}
See \cite{zor}.
\end{proof}
\end{proposition}
Notice that Galanopoulos in \cite{Gal}, Theorem(1.3) has shown  the existence of $\psi\in M(\mathcal{B}_p)$  such that $\int_{\mathbb{U}}(1-|z|^2)^{p-2}|\psi'(z)|^p(\log\frac{2}{1-|z|^2})^{p-1}dA(z)=\infty,$ i.e. the reverse of (ii) in preceding proposition is not correct. 

Let  $L(X)$ be the space of all linear bounded operators from locally convex Hausdorff  space $X$ into itself and $T \in L(X)$, the  Ces$\acute{a}$ro means of $T$ is defined by
$$T_{[n]}:=\frac{1}{n}\sum_{m=1}^nT^m, \ n \in \mathbb{N}.$$

An operator $T$ is (uniformly) mean ergodic if $\{{T_{[n]}}\}_{n=0}^{\infty}$ is a  convergent sequence in (norm) strong topology and is called $power\ bounded $ if the sequence $\{{T^n}\}_{n=0}^{\infty}$ is bounded in $L(X)$.\\ 
It is easy to check that for all $n\in\mathbb{N}$, $\frac{1}{n}T^{n}=T_{[n]}-\frac{n-1}{n}T_{[n-1]}$, where $T_{[0]}=I$ is the identity operator. From this we get if $T$ is mean ergodic, then for all $x\in X$, $\lim_{n\rightarrow \infty}\frac{1}{n}T^{n}x=0$ and in the uniform mean ergodic case, $\lim_{n\rightarrow \infty}\frac{1}{n} ||T^{n}||=0$.
The study of mean ergodicity of linear operators on Banach spaces goes back to 1931, when Von Numann proved that for a unitary operator $T$ on a Hilbert space $H$, there is a projection $P$ on $H$, such that $T_{[n]}$ converges to $P$ in the strong operator topology. 	In 1939 Lorch  demonstrated that for reflexive Banach spaces, power bounded operators are mean ergodic. Dunford in 1943 stated the connection between the spectral properties of an operator and its uniform mean ergodicity.
There are a lot of references about dynamical properties of different linear bounded operators on Banach, Fr$\acute{e}$chet and locally convex spaces. One of the best, is a book written by Bayart and Matheron \cite{Bay}. Additionally, \cite{Alb1, Alb2, kre} and the references therein give more details about mean ergodic and power bounded operators on locally convex spaces.
Bonet and Ricker \cite{Bon2}, characterized the mean ergodicity of multiplication operators in weighted spaces of holomorphic functions and recently Bonet, Jord$\acute{a}$ and Rodríguez \cite{Bon1} extended the results to the weighted space of continuous functions.
In this paper, we look for conditions  under which the multiplication operator $M_{ \psi}$ is power bounded and its  Ces$\acute{a}$ro means   is convergent  or uniformly convergent on the Bloch space  $\mathcal{B}$,  little Bloch space  $\mathcal{B}_0$ and the Besov Space $\mathcal{B}_p$.
\section{The  Bloch space  $\mathcal{B}$ and The  little Bloch space  $\mathcal{B}_0$}
The following Proposition provides  the necessary condition for multiplication operators to be power bounded, uniform mean ergodic and mean ergodic on $\mathcal{B}$ or $\mathcal{B}_0$.  
\begin{proposition}
	Suppose $\psi\in H(\mathbb{U})$ and $M_\psi$ is a bounded operator on $\mathcal{B}$($\mathcal{B}_0$). If $M_\psi$ is power bounded, mean ergodic or uniformly mean ergodic on $\mathcal{B}$($\mathcal{B}_0$), then $||\psi||_\infty\le1$. 
	\begin{proof}
		Suppose $M_\psi$ is uniformly mean ergodic or mean ergodic on $\mathcal{B}(\mathcal{B}_0$), then for all $f\in\mathcal{B}(\mathcal{B}_0$), we have $\lim_{n\rightarrow \infty}\frac{1}{n}f.\psi^n=0$ when $n\rightarrow \infty$. Let $f\equiv 1$, so $||\frac{\psi^n}{n}||_\mathcal{B} \rightarrow 0.$ By lemma (1.1), for all $z\in \mathbb{U}$, $|\frac{\psi^n(z)}{n}|\le||\frac{\psi^n}{n}||_\mathcal{B}\log\frac{2}{1-|z|^2},$ So for all $z\in \mathbb{U}$, $|\frac{\psi^n(z)}{n}|\rightarrow 0$ as $n\rightarrow \infty,$  it forces $|\psi(z)|\le1$ for all $z\in \mathbb{U}$ or equivalently, $||\psi||_\infty\le1$.

		Now suppose $||\psi||_{\infty} >1$. Choose $\lambda\in \mathbb{R}$ such that $||\psi||_{\infty} >\lambda > 1$. By (1.1), in both $\mathcal{B}$ and $\mathcal{B}_0$ we have  $||M_{\psi^n}|| \ge ||\psi^n||_{\infty}=||\psi||^n_{\infty} > \lambda^n$. $\lim_{n\rightarrow \infty}\lambda^n=\infty$, therefore $\{M_{\psi^n}\}_n$ is an unbounded sequence and consequently  $M_{\psi}$  is not power bounded operator on $\mathcal{B}$($\mathcal{B}_0)$.
	\end{proof}
\end{proposition}
The following Lemma which is proved in \cite{brown} is very useful in proving the next theorem. We omit the proof. 
\begin{lemma}{(See \cite{brown})}
	If $\{f_{n}\}\subseteq \mathcal{B}_0$ then $f_{n}\rightarrow 0$ weakly if and only if $f_{n}(z)\rightarrow 0$ for all $z\in\mathbb{U}$ and $\sup_{n}||f_{n}||_{\mathcal{B}}<\infty$. 
\end{lemma}
\begin{remark}
Consider that for $\psi\in H^\infty(\mathbb{U})$ Schwarz-Pick Lemma \cite{Bay} implies that $\beta_\psi\leq ||\psi||_\infty$. In fact if $\psi\not\equiv 0$, then for $\varphi=\frac{\psi}{||\psi||}$ we have:
$$\frac{\beta_\psi}{||\psi||_{\infty}}=\beta_\varphi=\sup_{z\in\mathbb{U}}(1-|z|^2)|\varphi'(z)|\leq\sup_{z\in\mathbb{U}}(1-|\varphi(z)|^2)\leq 1.$$
\end{remark}
\begin{theorem}
Suppose  $\psi$ is a holomorphic  function on $\mathbb{U}$ inducing a bounded multiplication operator $M_{\psi}$  on $\mathcal{B}$($\mathcal{B}_0$), if $||\psi||_{\infty}\leq 1$ then $M_{\psi}$ is power bounded on $\mathcal{B}$($\mathcal{B}_0)$.
\begin{proof}
We recall that $\mathcal{B}_0^{**}=\mathcal{B}$ and $(M_{\psi}|_{\mathcal{B}_0})^{**}=M_{\psi}|_{\mathcal{B}}$. (By $M_{\psi}|_{\mathcal{B}_0}$ and $M_{\psi}|_{\mathcal{B}}$ we mean the multiplication on $\mathcal{B}_{0}$ and $\mathcal{B}$, respectively). So $||(M_{\psi}|_{\mathcal{B}_0})||=||M_{\psi}|_{\mathcal{B}}||$ and clearly power boundedness of $M_{\psi}$ on $ \mathcal{B}$ implies it on $\mathcal{B}_0$ and vice versa.
By (1.2), for all $n\in\mathbb{N}$, $||M_{\psi^n}|| \le max\{||\psi^n||_{\mathcal{B}},||\psi^n||_{\infty} +\sigma_{\psi^n}\}$. If for some $z\in\mathbb{U}$, $|\psi(z) |=1$, then by Maximum Modules Principal, $ \psi\equiv \xi$ where $|\xi|=1$, then $\sigma_{\psi^n}=0$ for all $n\in\mathbb{N}$. So for $n\in\mathbb{N}$, $||M_{\psi^n}||\le||\psi^n||_\mathcal{B}=|\psi^n(0)|+\beta_\psi^n\le|\psi^n(0)|+||\psi||_\infty^n\le2$ or $||M_{\psi^n}||\le||\psi||_\infty^n+\sigma_\psi^n\le1$, then $\sup_{n\in\mathbb{N}}||M_{\psi^n}||\le2$ and $M_\psi$  is power bounded.\\
Now, suppose for all $z\in\mathbb{U}$, $|\psi(z)|<1.$ Fix $f\in\mathcal{B}_0$ and let $L:\mathcal{B}_0\rightarrow\mathbb{C}$ be a bounded linear functional. There exists $g\in A^{1}(\mathbb{U})$ such that $L(M_{\psi^{n}}f)=\int_{\mathbb{U}}\psi^{n}(z )f(z)\overline{g(z)}dA(z)$. Clearly, for all $n\in\mathbb{N}$ the function $\psi^{n}(z)f(z)\overline{g(z)}$ is integrable and Since $|\psi^{n}(z)|<1$, it converges pointwise to zero and by Lebesgue Convergence Theorem, $L(M_{\psi^{n}}f)\rightarrow 0$ and by Lemma 2.2 and Uniform boundedness Principle the result follows.\\
\end{proof}
\begin{remark} In the case $||\psi||_{\infty}<1$, we can find a upper bound for the sequence $\{M_{\psi^{n}}\}$.  If for $n\in\mathbb{N}$, $||M_{\psi^n}||\le||\psi^n||_{\mathcal{B}}$, then
	\begin{align}
		||M_{\psi^n}||&\le||\psi^n||_{\mathcal{B}}=|\psi^n(0)|+\beta_{\psi^n}\notag\\
		&\le|\psi^n(0)|+||\psi^n||_{\infty}\notag\\
		&=|\psi^n(0)|+||\psi||^n_{\infty}\le2,
	\end{align}
	and if $n\in\mathbb{N}$,  $||M_{\psi^n}|| \le ||\psi^n||_{\infty} +\sigma_{\psi^n}$, then
	\begin{align}
		||M_{\psi^n}||&\le||\psi^n||_{\infty}+\frac{n}{2}\sup_{z\in\mathbb{U}}(1-|z|^2)|\psi'(z)||\psi(z)|^{n-1}\log\frac{1+|z|}{1-|z|}\notag\\
		&\le||\psi^n||_\infty+n||\psi||^{n-1}_{\infty}\left(\frac{1}{2}\sup_{z\in\mathbb{U}}(1-|z|^2)|\psi'(z)|\log\frac{1+|z|}{1-|z|}\right)\notag\\
		&=||\psi||^n_\infty+n||\psi||^{n-1}_{\infty}\sigma_{\psi}\notag\\
		&\le1+n||\psi||^{n-1}_{\infty}\sigma_\psi.
	\end{align}
	But since $||\psi||_{\infty}<1$, we have $\lim_{n\rightarrow\infty}n||\psi||^{n-1}_{\infty}\sigma_\psi=0$ and therefore $\{n||\psi||^{n-1}_{\infty}\}$ is a bounded sequence. Let $K=\sup_{n\in\mathbb{N}}n||\psi||^{n-1}_{\infty}\sigma_\psi$, then $||M_{\psi^n}||\le1+K.$  So by this and (2.1), $\sup_{n\in\mathbb{N}}||M_{\psi^n}|| \le max\{2,1+K\}$.
\end{remark}
\end{theorem}
\begin{theorem}
Suppose that  $M_{\psi}$  is a bounded operator on $\mathcal{B}$($\mathcal{B}_0$). If $||\psi||_{\infty} < 1$ or $\psi\equiv\xi$ where $\xi\in\partial\mathbb{U}$, then $M_{\psi}$ is uniformly mean ergodic (and hence mean ergodic) on $\mathcal{B}$($\mathcal{B}_0$).
\end{theorem}  
\begin{proof}
	Primarily suppose $\psi\equiv\xi$, $|\xi|=1$. If $\xi=1$, then for all $n \in \mathbb{N}$ and $f\in\mathcal{B}$, $(M_{\psi})_{[n]}f=f$ and clearly $M_{\psi}$ is uniformly mean ergodic on $\mathcal{B}$($\mathcal{B}_0$) and if $\xi\ne1$, then $M_{\psi^n}f=\psi^nf=\xi^nf$ and $(M_{\psi})_{[n]}f=\frac{\xi+\xi^2+\dots+\xi^n}{n}f=\frac{f}{n}\frac{\xi(1-\xi^{n+1})}{1-\xi}$.

In this case for $f\in \mathcal{B}$($\mathcal{B}_0)$ with $||f||_{\mathcal{B}}\le1$, we have
\begin{align*}
	||(M_{\psi})_{[n]}f||_{\mathcal{B}}&=|\frac{f(0)}{n}\frac{\xi(1-\xi^{n+1})}{1-\xi}|+\frac{1}{n}\sup_{z\in\mathbb{U}}(1-|z|^2)|f'(z)||\frac{1-\xi^{n+1}}{1-\xi}|\\
	&\le\frac{2}{n|1-\xi|}+||f||_{\mathcal{B}}\frac{2}{n|1-\xi|}\le\frac{4}{n|1-\xi|},
\end{align*}
so $||(M_{\psi})_{[n]}|| \rightarrow 0$ when $n\rightarrow \infty$ and $M_{\psi}$ is uniformly mean ergodic on $\mathcal{B}$($\mathcal{B}_0$).

Now, suppose $||\psi||_{\infty}<1$ and let $f\in\mathcal{B}$($\mathcal{B}_0$) such that $||f||_{\mathcal{B}}\le1$, then 
\begin{align}
||(M_{\psi})_{[n]}f||_{\mathcal{B}}&\le\frac{|f(0)|}{n}\sum_{m=1}^{n}|\psi(0)|^m+
\frac{1}{n}\sup_{z\in\mathbb{U}}(1-|z|^2)\left|f'(z)\psi(z)\frac{1-\psi^{n}(z)}{1-\psi(z)}\right|+\notag\\
&\frac{1}{n}\sup_{z\in\mathbb{U}}(1-|z|^2)\left|f(z)\left(\psi(z)\frac{1-\psi^{n}(z)}{1-\psi(z)}\right)'\right|.
\end{align}
We show that all three components of the right side are uniformly convergent to 0 and so the proof will be complete. Since $||\psi||_{\infty}<1$, Maximum Modules principle implies $|\psi(0)|<1$, so $\{\frac{|f(0)|}{n}\sum_{m=1}^{n}|\psi(0)|^m\}_n$ converges to zero, as  $n\rightarrow \infty.$ Similarly, $\frac{1}{n}\sup_{z\in\mathbb{U}}(1-|z|^2|)|f'(z)\psi(z)\frac{1-\psi^{n}(z)}{1-\psi(z)}| \le\frac{1}{n}\frac{2}{1-||\psi||_{\infty}},$
therefore $\{\frac{1}{n}\sup_{z\in\mathbb{U}}(1-|z|^2|)|f'(z)\psi(z)\frac{1-\psi^{n}(z)}{1-\psi(z)}|\}_n$ converges uniformly to 0 as $n\rightarrow \infty.$ Finally, since by lemma (1.1), $|f(z)|\le||f||_{\mathcal{B}}\log\frac{2}{1-|z|^2}$, we get: 
\begin{align*} &\sup_{z\in\mathbb{U}}(1-|z|^2)|f(z)|\left|\left(\psi(z)\frac{1-\psi^{n}(z)}{1-\psi(z)}\right)'\right|\\
\le&\sup_{z\in\mathbb{U}}(1-|z|^2)|f(z)|\left|\psi'(z)\frac{1-\psi^{n}(z)}{1-\psi(z)}\right|\\
+&\sup_{z\in\mathbb{U}}(1-|z|^2)|f(z)|\left|\psi(z)\left(\frac{1-\psi^{n}(z)}{1-\psi(z)}\right)'\right|\\
\le&\sup_{z\in\mathbb{U}}(1-|z|^2)||f||_{\mathcal{B}}\log\frac{2}{1-|z|^2}||\psi'(z)|\frac{2}{(1-||\psi||_{\infty})^2}\\	
	 +&\sup_{z\in\mathbb{U}}(1-|z|^2)||f||_{\mathcal{B}}\log\frac{2}{1-|z|^2}||\psi'(z)|\frac{1+(n-1)||\psi||_{\infty}^{n+1}+n||\psi||_{\infty}^{n}}{(1-||\psi||_{\infty})^2}\\
\le&||f||_{\mathcal{B}}\left(\sup_{z\in\mathbb{U}}(1-|z|^2)|\psi'(z)|\log\frac{2}{1-|z|^2}\right) \left(\frac{3+(n-1)||\psi||_{\infty}^{n+1}+n||\psi||_{\infty}^{n}}{(1-||\psi||_{\infty})^2}\right)
\end{align*}
but $M_{\psi}$ is bounded operator on $\mathcal{B}$, so by \cite{brown} $\sup_{z\in\mathbb{U}}(1-|z|^2)|\psi'(z)|\log\frac{2}{1-|z|^2}<\infty$,
then the last part of the inequality (2.3) tends also to the zero
, when $n\rightarrow \infty$ and there is nothing left to prove. 
\end{proof}
\begin{lemma}
	 Let $M_{\psi}$ be a bounded operator on $\mathcal{B}$ $(\mathcal{B}_0)$, then $I-M_{\psi}$ is an isomorphism of $\mathcal{B}$ $(\mathcal{B}_0)$ if and only if $\frac{1}{1-\psi}\in H^{\infty}(\mathbb{U})$. 
\end{lemma}
\begin{proof}
If $I-M_{\psi}$ is invertible, then clearly, $(I-M_{\psi})^{-1}=(M_{1-\psi})^{-1}=M_\frac{1}{1-\psi}$, so by \cite{Arz2} we must have $\frac{1}{1-\psi}\in H^{\infty}(\mathbb{U})$. Conversely, suppose $\frac{1}{1-\psi}\in H^{\infty}(\mathbb{U})$. $$\sigma_\frac{1}{1-\psi}=\sup_{z\in \mathbb{U}}\frac{1}{2}(1-|z|^{2})\frac{|{\psi}'(z)|}{|1-\psi(z)|^{2}} \log\frac{1+|z|}{1-|z|}\leq ||\frac{1}{1-\psi}||_{\infty}  \, \,\sigma_{\psi}<\infty,$$ so by \cite{Arz2}, $M_{\frac{1}{1-\psi}}$ is bounded on $\mathcal{B}$ $(\mathcal{B}_0)$, i. e. $I-M_{\psi}$ is invertible.
\end{proof} 
\begin{proposition}
	Let $\psi$ be a non constant analytic function on $\mathbb{U}$ with $||\psi||_\infty =1$. If $M_{\psi}$ is bounded on $\mathcal{B}_0$ then it is mean ergodic. Moreover, it is uniformly mean ergodic on $\mathcal{B}_0$ if and only if $\frac{1}{1-\psi}\in  H^{\infty}(\mathbb{U})$.
\end{proposition}
\begin{proof}
Fix $f\in\mathcal{B}_0$. By Theorem 2.4, $M_{\psi}$ is power bounded. Let $M=\sup_{n\in\mathbb{N}}||M_{\psi^{n}}||$. Then $||(M_{\psi})_{[n]}||\leq M$ and also $\lim_{n\rightarrow \infty}\frac{1}{n}M_{\psi^{n}}f=0$. Maximum Modules Principle implies that for all $z\in\mathbb{U}$, $|\psi(z)|<1$, so $\{(M_{\psi})_{[n]}\}_n$ is a bounded sequence that converges pointwise to zero. We show it converges weakly to zero. Let $L\in\mathcal{B}_0^{*}$, fore some $g\in A^{1}(\mathbb{U})$, $L((M_{\psi})_{[n]}f) =\int_{\mathbb{U}}(M_{\psi})_{[n]}f(z)\overline{g(z)}dA(z)$. By Lebesgue Convergence Theorem we can deduce that $L((M_{\psi})_{[n]}f)\rightarrow 0$. The proof of proposition is completed by using Theorem 1.1, page 72 of \cite{kre}.\\
	Suppose $M_{\psi}$ is uniformly mean ergodic. Via the above proof we must have $||(M_{\psi})_{[n]}||\rightarrow 0$. One can easily see that $ker(I-M_{\psi})=0$ and since $M_{\psi}$ is power bounded, $\frac{1}{n}||M_{\psi^{n}}||\rightarrow 0$. So by proposition 2.16 of \cite{Alb1}, $M_{\psi}$ is uniformly mean ergodic if and only if $I-M_{\psi}=M_{1-\psi}$ is an isomorphism of $\mathcal{B}_0$, which is equivalent to  $\frac{1}{1-\psi}\in H^{\infty}(\mathbb{U})$.
\end{proof}
\begin{proposition}
	Suppose $\psi\in H(\mathbb{U})$ is non constant and $||\psi||_\infty=1$ and $M_\psi$ is a bounded operator on  $\mathcal{B}$, then $M_\psi$ is mean ergodic if and only if it is uniformly mean ergodic if and only if  $\frac{1}{1-\psi}\in H^{\infty}(\mathbb{U})$. 
\end{proposition}
\begin{proof}
As we said before, $||(M_{\psi})_{[n]}|_{\mathcal{B}}||=||(M_{\psi})_{[n]}|_{\mathcal{B}_{0}}||$, so by above proposition $M_\psi$  is uniformly mean ergodic if and only if  $\frac{1}{1-\psi}\in H^{\infty}(\mathbb{U})$.  On the other hand Bloch space is a Grothendieck Banach space which satisfies the Dunford Pettis property (GDP space) which Lotz in \cite{lotz} proved that mean ergodicity and uniform mean ergodicity are equivalent in these spaces. See \cite{jorda} page 16.
\end{proof}
The forthcoming example is a direct consequence of previous propositions. 
\begin{example}
	
	$M_{z}$ is power bounded operator  on both $\mathcal{B}$ and $\mathcal{B}_0$. But it is  not uniformly mean ergodic nor is it mean ergodic on $\mathcal{B}$. It is mean ergodic on $\mathcal{B}_0$, but it is not uniformly mean ergodic. These statement are    also true for $M_{\psi}$, where $\psi$ is an automorphism of the unit disk.  
\end{example}
The following theorems are direct consequences of this sections:
\begin{theorem}
Let $M_{\psi}$ be a bounded multiplication operator on $\mathcal{B}$, then the following are equivalent:
\begin{enumerate}
	 \item  $||\psi||_{\infty} \leq 1$ and either, $\psi\equiv\xi,$ where $\xi\in\partial\mathbb{U}$ or $\frac{1}{1-\psi}\in H^{\infty}(\mathbb{U})$. 
	\item  $M_{\psi}$ is mean ergodic.
	\item  $M_{\psi}$ is uniformly mean ergodic.
\end{enumerate} 
\end{theorem}
\begin{theorem}
	Let $M_{\psi}$ be a bounded multiplication operator on $\mathcal{B}_0$, then the following are equivalent:
	\begin{enumerate}
		\item $||\psi||_{\infty} \leq 1$
		\item $M_{\psi}$ is power bounded.
		\item $M_{\psi}$ is mean ergodic. 
	\end{enumerate}
\end{theorem}
\begin{theorem}
	Let $M_{\psi}$ be a bounded multiplication operator on $\mathcal{B}_0$, then the following are equivalent:
	\begin{enumerate}
		\item  $||\psi||_{\infty} \leq 1$ and either, $\psi\equiv\xi,$ where $\xi\in\partial\mathbb{U}$ or $\frac{1}{1-\psi}\in H^{\infty}(\mathbb{U})$. 
		\item  $M_{\psi}$ is uniformly mean ergodic.
	\end{enumerate} 
\end{theorem}   
\section{Besov Space $\mathcal{B}_p (1<p<\infty)$}
Before starting this section, it is necessary to remind that a Banach space $X$ is said to be $mean \ ergodic$ if each power bounded operator is mean ergodic. Lorch  by extending  the result of Rizes, showing that $L_p$ spaces are mean ergodic, proved that the reflexive spaces are also mean ergodic, see \cite{Alb1}. According to the introduction, for $1<p<\infty$ Besov Spaces $\mathcal{B}_p$ are reflexive spaces and therefore power boundedness  of an operator implies mean ergodicity. 
In this section we only consider the case $1<p<\infty$.
\begin{theorem}
	Suppose $\psi\in H(\mathbb{U})$ and $M_\psi$ is a bounded operator on Besov space $\mathcal{B}_p$. If $M_\psi$ is power bounded, mean ergodic or uniformly mean ergodic operator on $\mathcal{B}_p$, then $||\psi||_\infty\le1$. 
	\begin{proof}
		
		First suppose $||\psi||_{\infty}>1$. Then there is $\alpha>0$ such that $||\psi||_{\infty}>\alpha>1.$ Since by (1.1), $||\psi||_\infty \le ||M_{\psi}||$, for all $n\in \mathbb{N}$. we have $\alpha^n<||\psi^n||_{\infty}\le ||M_{\psi^n}||$  and therefore $M_{\psi}$ can not be power bounded on $\mathcal{B}_p$.
		Now suppose $M_\psi$ is uniformly mean ergodic (or mean ergodic) on $\mathcal{B}_p$, then for all $f\in\mathcal{B}_p$ we have $\lim_{n\rightarrow \infty}\frac{1}{n}f.\psi^n=0$ when $n\rightarrow \infty$. Let $f\equiv 1$, then $||\frac{\psi^n}{n}||_{p} \rightarrow 0$. By lemma(1.2), $|\frac{\psi^n(z)}{n}|\le C||\frac{\psi^n}{n}||_{p}(\log\frac{2}{1-|z|^2})^{1-\frac{1}{p}}$ for all $z\in \mathbb{U}$ and some $C\ge 0$. So $|\frac{\psi^n(z)}{n}|\rightarrow 0$ when $n\rightarrow \infty$ for all $z\in \mathbb{U}$,so $|\psi(z)|\le1$ for all $z\in \mathbb{U}$ and finally $||\psi||_\infty\le1$.

	\end{proof}
\end{theorem}
From now on, we assume that analytic function $\psi$ holds in the following condition:
\begin{align}
\int_{\mathbb{U}}(1-|z|^2)^{p-2}|\psi'(z)|^p(\log\frac{2}{1-|z|^2})^{p-1}dA(z)<\infty.
\end{align} 
\begin{theorem}
Suppose that $\psi \in H(\mathbb{U})$ and condition (3.1) is met. If $M_{\psi}$  is a bounded operator on the Besov space $\mathcal{B}_p$, then the following statements are equivalent.
\begin{itemize}
\item[(i)] $||\psi||_{\infty} \le 1$. 
\item[(ii)]  $M_{\psi}$  is power bounded.
\item[(iii)] $M_{\psi}$ is mean ergodic.
\end{itemize}
\begin{proof}

According to the initial interpretations of the section, it is sufficient to show that (i) and (ii) are equivalent.
 Let $||\psi||_{\infty}\le1$. If there exist $z\in \mathbb{U}$ such that $|\psi(z)|=1$, then $\psi(z)=\lambda$, $|\lambda|=1$ and $\psi'\equiv 0$. So for $f\in \mathcal{B}_p$ and $||f||_{p}=1$ we have
\begin{align*}
||M_{\psi^n}f||_{p}&=|\psi^n(0)f(0)|+\gamma_{\psi^nf}\\
&\le|f(0)|+\gamma_f=||f||_{p},
\end{align*}
and $M_{\psi}$ is power bounded on $\mathcal{B}_p$, in fact $||M_{\psi^n}||\le1$, for all $n\in\mathbb{N}$. Now suppose $|\psi(z)|<1$ for all $z\in \mathbb{U}$ and let $f\in\mathcal{B}_p.$ In this case, the followings can be deduced;
\begin{itemize}
\item[1.] $|\psi^n(0)f(0)| \rightarrow 0$ when $n\rightarrow \infty,$ since $|\psi(0)|<1.$
\item[2.] $\int_{\mathbb{U}}|f'(z)|^p|\psi^n(z)|^p(1-|z|^2)^{p-2}dA(z) \rightarrow 0$, as $n\rightarrow \infty,$  since $|f'(z)\psi^n(z)|^p(1-|z|^2)^{p-2}\le |f'(z)^p(1-|z|^2)^{p-2},$ and $f\in\mathcal{B}_p$ gives us that  $\int_{\mathbb{U}}|f'(z)|^p(1-|z|^2)^{p-2}dA(z)<\infty,$ then $|f'(z)\psi^n(z)|^p(1-|z|^2)^{p-2}$ is integrable for all $n\in\mathbb{N}.$ By using Lebesgue Convergence theorem the result is obtained. 
\item[3.] $\int_{\mathbb{U}}|n\psi'(z)\psi^{n-1}(z)f(z)|^p(1-|z|^2)^{p-2}dA(z)\rightarrow 0,$ since by lemma (1.2): 
$$|n\psi'(z)\psi^{n-1}(z)f(z)|^p(1-|z|^2)^{p-2}\le n^p C^p||f||_{p}^p (1-|z|^2)^{p-2}|\psi'(z)|^p (\log\frac{2}{1-|z|^2})^{p-1},$$ by hypothesis the right side of the last inequality is integrable for all $n\in\mathbb{N}$ and so is $|n\psi'(z)\psi^{n-1}(z)f(z)|^p(1-|z|^2)^{p-2}$. Lebesgue Converges theorem gives the desired result.
\end{itemize}
Consequently for all $f\in\mathcal{B}_p$, $||M_{\psi^n}f||_{p} \rightarrow 0 $ when $n\rightarrow \infty$. So $\{M_{\psi^n}f\}$ is bounded sequence for all $f\in\mathcal{B}_p$ and by Principle uniform boundedness $M_{\psi}$ is power bounded on $\mathcal{B}_p$. 
\end{proof}
\end{theorem}
Recall that by $\sigma(T)$ (spectrum of $T$) we mean the set of all $\lambda\in\mathbb{C}$ such that $T-\lambda I$ is not invertible.
\begin{lemma}
Suppose $\psi \in H(\mathbb{U})$ which satisfies condition (3.1) 
and $M_{\psi}$  is a bounded operator on the Besov space $\mathcal{B}_p$, then $\overline{\psi(\mathbb{U})}= \sigma(M_\psi),$ ($\overline{\psi(\mathbb{U})}$ means the norm closure of $\psi(\mathbb{U})$).
\begin{proof}
First since $M_\psi-\lambda I=M_{\psi-\lambda}$, then $\lambda\in \sigma(M_\psi)$ if and only if $M_{\psi-\lambda}$ is not invertible. If $M_{\psi-\lambda}$ is invertible, then $(M_{\psi-\lambda})^{-1}=M_{(\psi-\lambda)^{-1}}=M_{\frac{1}{\psi-\lambda}}$. So if $\lambda \in \psi (\mathbb{U})$ then there exists $z_0\in \mathbb{U}$ such that $\psi(z_0)=\lambda$ therefore $\frac{1}{\psi-\lambda}\notin H^{\infty}(\mathbb{U})$ and $M_{\psi-\lambda}$ is not invertible that means $\lambda \in \sigma(M_\psi)$ and $\psi(\mathbb{U}) \subseteq \sigma(M_\psi)$. But $\sigma(M_\psi)$ is closed so $\overline{\psi(\mathbb{U})}\subseteq \sigma(M_\psi)$. Now assume that (3.1) holds and $\lambda\notin \overline{\psi(\mathbb{U})}$, hence $\frac{1}{\psi(z)-\lambda}\in H^\infty(\mathbb{U})$. By (3.1) 
$$\int_{\mathbb{U}}\frac{|\psi'(z)|^p}{|\psi(z)-\lambda|^{2p}}\log\frac{2}{1-|z|^2})^{p-1}(1-|z|^2)^{p-2}dA(z)<\infty.$$
Thus by proposition (1.4), $M_{\frac{1}{\psi-\lambda}}$ is bounded on $\mathcal{B}_p$ and  $M_{\psi-\lambda}$ is invertible which means $\lambda \notin \sigma(M_\psi)$.

\end{proof} 
\end{lemma}
Dunford in \cite{Du} stated the connection between the spectral properties of an operator and its uniform mean ergodicity. 
The following Theorem  represents Lin and Dunford Theorems together.
\begin{theorem}
	If an operator $T$ on a Banach space $X$ is uniformly mean ergodic, if and only if both $(||T^n||/n)_n$ converges to $0$ and either $1\in \mathbb{C}\setminus \sigma(T)$ or $1$ is a pole of order $1$ of the resolvent $R_T:\mathbb{C}\setminus\sigma(T)\rightarrow L(X), R_T(\lambda)=(T-\lambda I)^{-1}$. Consequently if $1$ is an accumulation of $\sigma(T)$, then $T$ is not uniformly mean ergodic.
\end{theorem}
\begin{proof}
	See \cite{Du} and \cite{lin}.
\end{proof}
\begin{theorem}
Suppose $\psi \in H(\mathbb{U})$ which holds $(3.1)$ and $M_{\psi}$  is a bounded operator on the Besov space $\mathcal{B}_p$, then $M_{\psi}$ is uniformly mean ergodic on $\mathcal{B}_p$ if and only if $||\psi||_{\infty}\le1$ and either $\psi\equiv \xi$ for some $\xi\in\partial\mathbb{U}$ or $\frac{1}{1-\psi}\in H^{\infty}(\mathbb{U}) $.
\begin{proof} 
Let $||\psi||_\infty\le1.$ Consider that $(M_\psi)_{[n]}f(z)=\frac{f(z)}{n}\sum_{m=1}^{n}(\psi(z))^n$.
 So if $\psi \equiv 1$, we can easily see $||(M_\psi)_{[n]}-I||\rightarrow 0$ when $n\rightarrow \infty$, where $I$ is the identity operator on $\mathcal{B}_p$. In the case $\psi\equiv \xi$, where $\xi\ne1$, we have $(M_{\psi})_{[n]}=\frac{\xi+\xi^2+\dots+\xi^n}{n}f=\frac{f}{n}\frac{\xi(1-\xi^{n+1})}{1-\xi}$ and clearly $||(M_{\psi})_{[n]} ||\rightarrow 0.$ If $\frac{1}{1-\psi}\in H^{\infty}(\mathbb{U}),$ an apply of propositon (1.4) shows that the function $\frac{1}{1-\psi}\in M(\mathcal{B}_p)$ and $M_\frac{1}{1-\psi}$ is bounded on $\mathcal{B}_p$, it means that $1\notin \sigma(M_\psi)$ and since $M_\psi$ is power bounded, Dunford-Lin Theorem guaranties the uniform mean ergodicity of $M_\psi$ on $\mathcal{B}_p$.

Conversely; assume that $M_\psi$ is uniformly mean ergodic on $\mathcal{B}_p$. So by Theorems (3.1) and (3.2)  it is power bounded and $||\psi||_\infty\le 1$. Suppose $\psi$ is not uni- modular constant function, so $|\psi(z)|<1$ for all $z\in\mathbb{U}$, this get us that $||(M_{\psi})_{[n]}||\rightarrow 0$ as $n\rightarrow\infty$. Also $ker(I-M_{\psi})={0}$ and power boundedness of $M_{\psi}$ implies that $\lim_{n\rightarrow \infty}\frac{1}{n}|| M_{\psi}^{n}||=0$, Proposition 2.16 of \cite{Alb1} confirms that $I-M_{\psi}$ is an isomorphism on $\mathcal{B}_p$, i.e. $M_{1-\psi}^{-1}=M_{\frac{1}{1-\psi}}$ is bounded on $\mathcal{B}_p$, thus $\frac{1}{1-\psi}\in H^{\infty}(\mathbb{U}) $.
\end{proof} 
\end{theorem}

Department of Mathematics, Faculty of Sciences, Shiraz Branch, Islamic Azad University,  Shiraz, Iran.\\
E-mail: ffalahat@gmail.com\\
Department of Mathematics, Faculty of Sciences, Shiraz Branch, Islamic Azad University,  Shiraz, Iran.\\
E-mail: zkamali@shirazu.ac.ir

\end{document}